\documentclass[11pt]{amsart}
\usepackage{geometry}                
\usepackage{float}
\usepackage{graphicx}
\usepackage{amsmath,amssymb,hyperref}

\newcommand{\R}{\mathbb{R}}

\newcommand{\C}{\mathbb{C}}

\newcommand{\re}{\operatorname{Re}}

\numberwithin{equation}{section}
\numberwithin{figure}{section}

\theoremstyle{plain} 
\newtheorem{theorem}{Theorem}[section]

\theoremstyle{definition} 

\title{A vase of catenoids}
\author{Peter Connor}
\address{Peter Connor\\Department of Mathematical Sciences\\Indiana University South Bend\\1700 Mishawaka Ave\\South
Bend\\IN 46634\\USA\\pconnor@iusb.edu}

\begin{document}
\maketitle

\noindent {\sc Abstract. } {\footnotesize}
In this note we construct a vase of catenoids - a symmetric immersed minimal surface with planar and catenoid ends.



\section{Introduction}
The building blocks for minimal surfaces with finite total curvature and embedded ends are planes and catenoids.  One can consider the possible arrangements of planar and catenoid ends that yields a minimal surface.  This note proves the existence of two beautiful families of minimal surfaces on punctured spheres.  The first, which we call a vase of catenoids, has a horizontal planar end, a downward pointing catenoid end with vertical normal, and symmetrically placed upward pointing catenoid ends with non-vertical normals.  See figure \ref{figure:kvase}.  
\begin{figure}[h]
	\centerline{ 
		\includegraphics[height=3.45in]{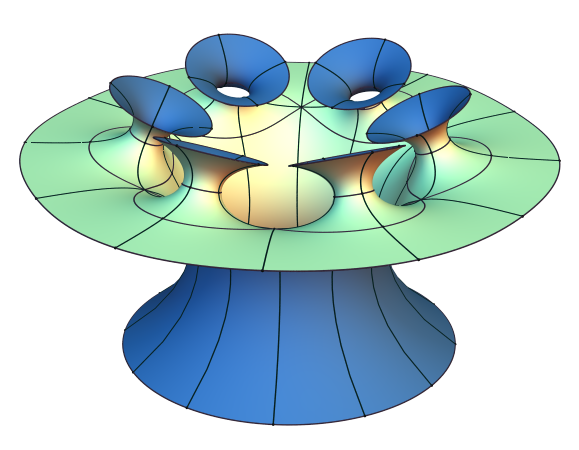}
			}
	\caption{Vase of catenoids}
	\label{figure:kvase}
\end{figure}

The second is a variation of the vase of catenoids.  Imagine taking two copies of a vase of catenoids.  Cut off the bottom catenoid end on each copy, and glue the copies together along the resulting closed curves.  See figure \ref{figure:doublekvase} for the resulting surface.  It has two horizontal planar ends, symmetrically placed downward pointing catenoid ends with non-vertical normals, and symmetrically placed upward pointing catenoid ends with non-vertical normals.
\begin{figure}[t]
	\centerline{ 
		\includegraphics[height=3.45in]{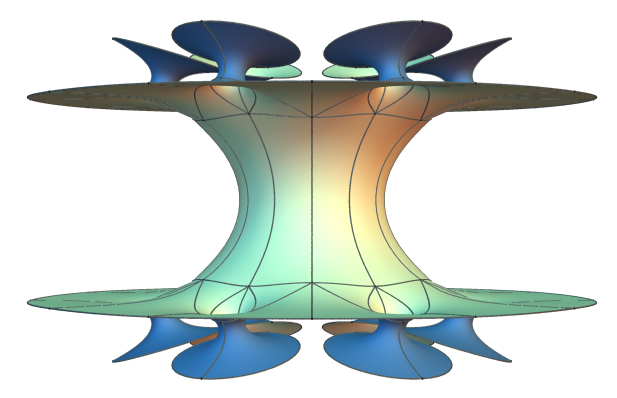}
			}
	\caption{Two vases glued together}
	\label{figure:doublekvase}
\end{figure}

The construction of these surfaces was inspired by the Finite Riemann minimal surface constructed in \cite{we5} with a horizontal planar end together with two catenoid ends with non-vertical normals, and also by the $k$-noid surface constructed by Jorge and Meeks \cite{jm1} with catenoid ends at each root of unity with horizontal normals in the direction of that root of unity.  None of these surfaces are embedded.  By the Lopez-Ros theorem \cite{lor1}, the plane and catenoid are the only complete embedded minimal surfaces on punctured spheres.  
\section{Weierstrass Representation}
We use the Weierstrass Representation of minimal surfaces on a punctured sphere, which may be written as
\[
X(z)=\re\int_{z_0}^z\left(\frac{1}{2}\left(\frac{1}{G}-G\right)dh,\frac{i}{2}\left(\frac{1}{G}+G\right)dh,dh\right)
\]
where $z,z_0\in\Sigma=\overline{\C}-\{p_1,p_2,\ldots,p_n\}$.  The points $p_1,p_2,\ldots,p_n$ are the ends of the surface, $G$ is the composition of stereographic projection with the Gauss map, and $dh$ is a meromorphic one-form called the height differential.  A good reference for the Weierstrass representation is \cite{os1}.  One issue is that $X$ depends on the path of integration.  The map $X:\Sigma\rightarrow\R^3$ is well defined provided that 
\[
\re\int_\gamma\left(\frac{1}{2}\left(\frac{1}{G}-G\right)dh,\frac{i}{2}\left(\frac{1}{G}+G\right)dh,dh\right)=(0,0,0)
\]
for all closed curves $\gamma$ in $\Sigma$.  This is called the period problem, and it can be expressed as
\[
\text{Res}_{p_j}\left(\left(\frac{1}{G}-G\right)dh\right)\in\R, \text{Res}_{p_j}\left(i\left(\frac{1}{G}+G\right)dh\right)\in\R, \text{Res}_{p_j}\left(dh\right)\in\R
\]
for $j=1,2,\ldots,n$.

A second issue is that we want $X$ to be regular.  This is ensured by requiring that $G$ has either a zero or pole at $p\in\Sigma$ if and only if $dh$ has a zero at $p$ with the same multiplicity.

\section{Constructions}
One can use the desired geometry of a minimal surface to create potential Weierstrass data for a minimal surface.  A vertical normal at $p_j$ corresponds to $G(p_j)=0$ (downward pointing normal) or $G(p_j)=\infty$ (upward pointing normal).  If $X$ has a horizontal planar end at $p_j$ then $G$ has a zero or pole of order $k+1$ and $dh$ has a zero of order $k-1$ at $p_j$.  If $X$ has a catenoid end at $p_j$ with vertical normal then $G$ has a simple pole or zero and $dh$ has a simple pole at $p_j$.  If $X$ has a catenoid end at $p_j$ with non-vertical normal then $G$ has neither a pole or zero and $dh$ has a double order pole at $p_j$.

When the domain is a punctured sphere, the sum of zeros of $G$ equals the sum of poles of $G$ and the sum of zeros of $dh$ is two less then the sum of poles of $dh$.

We can use the images of our desired surfaces to construct the Weierstrass data $G$ and $dh$.  For the vase of catenoids, place the horizontal planar end at $z=\infty$ and the downward pointing catenoid end with vertical normal at $z=0$.  Place the upward pointing catenoid ends with non-vertical normals at each root of unity.  Fix $G(0)=0$.  Then $G(\infty)=\infty$.  There will also be a point on each catenoid at the roots of unity with vertical downward pointing normal.  In keeping with the symmetry of the surface, fix
\[
G\left(ae^{i2\pi j/k}\right)=0
\]
for $j=0,1,\ldots,k-1$.
Let $G(z)$ have simple zeros at $z=0$ and the roots of $z^k=a^k$, with $a\in\R$.  Then, $G$ has a pole of order $k+1$ at $z=\infty$, and 
\[
G(z)=\rho z(z^k-a^k).
\]
The height differential has simple zeros at the roots of $z^k=a^k$, double order poles at the $k$th-roots of unity, and a simple pole at $z=0$.  This forces $dh$ to have a zero of order $k-1$ at $z=\infty$, and
\[
dh=\frac{a^k-z^k}{z(z^k-1)^2}dz.
\]

\begin{theorem}
For each positive integer $k>1$ and real number $a\in(0,1)$ there exists a $\rho>0$ such that
\[
\begin{split}
G(z)&=\rho z(z^k-a^k)\\
dh&=\frac{a^k-z^k}{z(z^k-1)^2}dz
\end{split}
\]
is the Weierstrass data for a minimal surface with a horizontal planar end at $z=\infty$, a downward pointing vertical catenoid end at $z=0$, and upward pointing non-vertical catenoid ends at the roots of unity.
\end{theorem}

\begin{proof}
All that remains is to solve the period problem. As the residues of $dh$ are all real, $dh$ has no periods.  The residues of $Gdh$ and $1/Gdh$ are zero at $z=0$ and $z=\infty$.  Assuming $\rho\in\R$, the symmetries of the surface we are constructing reduce the period problem to the equation
\[
0=\text{Res}_1\left(\left(\frac{1}{G}+G\right)dh\right)=\frac{\rho(a^k-1)(ka^k+k-a^k+1)}{k^2}+\frac{k+1}{\rho k^2}
\]
which is solved when 
\[
\rho=\sqrt{\frac{k+1}{(1-a^k)(ka^k+k-a^k+1)}}.
\]
\end{proof}

Examining figure \ref{figure:doublekvase}, the second family has horizontal planar ends at $z=0$ and $z=\infty$, downward pointing catenoid ends with non-vertical normal at the roots of $z^k=b^k$, and upward pointing catenoid ends with non-vertical normal at the roots of $z^k=1/b^k$.  The Gauss map is $0$ at the roots of $z^k=a^k$ and $\infty$ at the roots of $z^k=1/a^k$.  The height differential $dh$ has double order poles at the roots of $z^k=b^k$ and $z^k=1/b^k$ and simple zeros at the roots of $z^k=a^k$ and $z^k=1/a^k$.  In order for the surface to have horizontal planar ends at $0$ and $\infty$, we need $dh$ to have zeros at $0$ and $\infty$.  Thus, set $dh$ with zeros of order $k-1$ at $0$ and $\infty$.  This forces $G$ to have a zero of order $k+1$ at $z=0$ and a pole of order $k+1$ at $z=\infty$.  Hence,
\[
G(z)=\frac{\rho z^{k+1}(z^k-a^k)}{a^kz^k-1}
\]
and
\[
dh=\frac{b^{2k}z^{k-1}(z^k-a^k)(a^kz^k-1)}{a^k(z^k-b^k)^2(b^kz^k-1)^2}dz.
\]
If $\rho=1$ then, similar to the first example, the period problem reduces to a single equation.
\begin{theorem}
For each positive integer $k>1$ and real number $b\in(0,1)$ there exists an $a>0$ such that 
\[
\begin{split}
G(z)&=\frac{z^{k+1}(z^k-a^k)}{a^kz^k-1}\\
dh&=\frac{b^{2k}z^{k-1}(z^k-a^k)(a^kz^k-1)}{a^k(z^k-b^k)^2(b^kz^k-1)^2}dz
\end{split}
\]
is the Weierstrass data for a minimal surface with horizontal planar ends at $z=0$ and $z=\infty$, upward pointing non-vertical catenoid ends at solutions to $z^k=b^k$, and downward pointing non-vertical catenoid ends at solutions to $z^k=1/b^k$.
\end{theorem}
\begin{proof}
As with the vase of catenoids, the symmetries of the surface reduce the period problem to the equation

\[
\begin{split}
0=&\text{Res}_b\left(\frac{1}{G}+G\right)dh\\
=&\frac{b^{2k}\left(k-1+b^{2+2k}(k-1)+(b^2+b^{2k})(k+1)\right)a^{2k}}{a^kb(b^k-1)^3(b^k+1)^3k^2}\\
&+\frac{2b^k\left(1+b^{2+4k}-(b^{2k}+b^{2+2k})(2k+1)\right)a^k}{a^kb(b^k-1)^3(b^k+1)^3k^2}\\
&+\frac{-k-1+b^{2k}+b^{2+4k}-b^{2+6k}+3kb^{2k}+3kb^{2+4k}-kb^{2+6k}}{a^kb(b^k-1)^3(b^k+1)^3k^2}
\end{split}
\]

which is solved when

{\footnotesize
\[
a=\sqrt[k]{\frac{-1-b^{2+4k}+(b^{2k}+b^{2+2k})(2k+1)+(1-b^{2k})\sqrt{k^2+b^2(1-b^{2k})^2(2k+1)+k^2b^2(1+b^{4k}+b^{2+4k})}}{b^k(k-1+b^{2+2k}(k-1)+(b^2+b^{2k})(k+1))}}
\]}

and $0<b<1$.  We do need $a>0$.  Assume that $k>1$ and $0<b<1$.  Then the denominator in $a$ is clearly positive, and the numerator in $a$ is positive because it is greater than
\[
-1-b^{2+4k}+(b^{2k}+b^{2+2k})(2k+1)+(1-b^{2k})\sqrt{k^2}= k-1+b^{2k}(k+1-b^{2+2k}+b^2(2k+1))>0.
\]
Thus, $a>0$ when $k>1$ and $0<b<1$.  For example, $a\approx 3.97667$ when $b=0.25$ and $k=6$.  
\end{proof}
\addcontentsline{toc}{section}{References}
\bibliographystyle{plain}
\bibliography{minlit}

\begin{thebibliography}{1}

\bibitem{jm1}
L.~P. Jorge and W.~H. Meeks.
\newblock The topology of complete minimal surfaces of finite total gaussian
  curvature.
\newblock {\em Topology}, 22:203--221, 1983.

\bibitem{lor1}
F.~J. Lopez and A.~Ros.
\newblock On embedded complete minimal surfaces of genus zero.
\newblock {\em Journal of Differential Geometry}, 33(1):293--300, 1991.

\bibitem{os1}
R.~Osserman.
\newblock {\em A Survey of Minimal Surfaces}.
\newblock Dover Publications, New York, 2nd edition, 1986.

\bibitem{we5}
M.~Weber.
\newblock Classical minimal surfaces in euclidean space by examples. geometric
  and computational aspects of the weierstrass representation.
\newblock In {\em Global theory of minimal surfaces}, pages 19--64. American
  Mathematical Society, 2005.

\end{thebibliography}

\label{sec:liter}

\end{document}